\documentclass{amsart}

\usepackage{amsmath}
\usepackage{amssymb}
\usepackage{amsthm}
\usepackage{enumerate}

\usepackage{tikz}
\usetikzlibrary{arrows}
\newcommand{\midarrow}{\tikz \draw[red, -stealth'] (0,0) -- +(.2,0);}

\tikzset{vertex/.style={circle,draw,fill,inner sep=0pt,minimum size=1mm}}

\newtheorem{theorem}{Theorem}

\newtheorem{corollary}[theorem]{Corollary}
\newtheorem{definition}[theorem]{Definition}
\newtheorem{example}[theorem]{Example}
\newtheorem{lemma}[theorem]{Lemma}

\newtheorem{proposition}[theorem]{Proposition}
\newtheorem{remark}[theorem]{Remark}

\numberwithin{equation}{section}

\DeclareMathOperator{\tr}{tr}
\DeclareMathOperator{\Ima}{Im}
\DeclareMathOperator{\Ker}{Ker}

\newcommand{\AFix}{\mathrm{AFix}}

\mathchardef\mhyphen="2D
\newcommand{\nAFix}[1]{#1\mhyphen\AFix}

\newcommand{\Z}{\mathbb Z}

\begin{document}

\title{Lefschetz numbers and fixed point theory in digital topology}

\author{Muhammad Sirajo Abdullahi}

\curraddr[]{\textit{KMUTTFixed Point Research Laboratory},  KMUTT-Fixed Point Theory and Applications Research Group, SCL 802 Fixed Point Laboratory, Department of Mathematics, Faculty of Science, King Mongkut's University of Technology Thonburi (KMUTT), 126 Pracha-Uthit Road, Bang Mod, Thrung Khru, Bangkok 10140, Thailand}
\email{abdullahi.sirajo@udusok.edu.ng [M.S. Abdullahi]}
\email{poom.kumam@mail.kmutt.ac.th [P. Kumam]}

\author{Poom Kumam} 
\curraddr{Center of Excellence in \textit{Theoretical and Computational Science (TaCS-CoE)},  Science Laboratory Building, King Mongkut's University of Technology Thonburi (KMUTT), 126 Pracha-Uthit Road, Bang Mod, Thrung Khru, Bangkok 10140, Thailand}
\email{poom.kumam@mail.kmutt.ac.th [P. Kumam]}

\author{P. Christopher Staecker}
\curraddr{Department of Mathematics, Fairfield University, Fairfield, CT 06823-5195, USA}
\email{cstaecker@fairfield.edu [P.C. Staecker]}

\address[]{Department of Mathematics, Faculty of Science, Usmanu Danfodiyo University, Sokoto, Nigeria}
\email{abdullahi.sirajo@udusok.edu.ng [M.S. Abdullahi]}


\thanks{The first author was supported by the ``Petchra Pra Jom Klao Ph.D. Research Scholarship from King Mongkut's University of Technology Thonburi".}






\begin{abstract}
	In this paper, we present two types of Lefschetz numbers in the topology of digital images. Namely, the simplicial Lefschetz number $L(f)$ and the cubical Lefschetz number $\bar L(f)$. We show that $L(f)$ is a strong homotopy invariant and has an approximate fixed point theorem. On the other hand, we establish that $\bar L(f)$ is a homotopy invariant and has an $n$-approximate fixed point result. In essence, this means that the fixed point result for $L(f)$ is better than that for $\bar L(f)$ while the homotopy invariance of $\bar L(f)$ is better than that of $L(f)$. Unlike in classical topology, these Lefschetz numbers give lower bounds for the number of approximate fixed points. Finally, we construct some illustrative examples to demonstrate our results.
\end{abstract}

\maketitle

\section{Introduction}
The theory of fixed points plays a major role in many areas of mathematical research including functional analysis and applied topology. The Banach fixed point theorem \cite{Banach} is the pioneer result from the perspective of metric spaces, which guarantees not only the existence and uniqueness of a fixed point of a certain self-map $f$ but also provides a constructive method of finding such a fixed point. This inspired many researchers to focus on improvements or generalizations in many different directions (see \cite{AbdullahiAzam17A,AbdullahiPoom18,Azametal09,Nadler,SintuKumam11} and others).

In topological perspectives,  the Brouwer fixed point theorem \cite{Brouwer11abbildung}, the Lefschetz fixed point theorem \cite{Lefschetz26intersections} and the Nielsen fixed point theory \cite{Nielsen27untersuchungen} are the main fundamental fixed point results. These theorems tell us some information regarding the existence and number of fixed points of some certain self-map $f$, and the fixed points of any map $g$ homotopic to $f$ in a topological space $X$. 

Remember that in classical topology, we say that a space $X$ has the fixed point property (FPP, for short) if every continuous self-map $f$ on $X$ has a fixed point. A similar definition appeared in digital topology:
\begin{definition} \cite{Rosen86continuous}
	A digital image $(X, \kappa)$ has the FPP if every $\kappa$-continuous $f$ on $X$ has a fixed point.
\end{definition}

However, the FPP is essentially a vacuous property, because the digital image of a single point is the only digital image with the FPP as was shown below:
\begin{theorem}\cite{Boxeretal16digital}
	A digital image $(X, \kappa)$ has the FPP if and only if $\#X = 1.$
\end{theorem} 
Digital topology, though not technically a topological theory in the classical sense, has provided a theoretical formulation of numerous image processing operations like image thinning and segmentation, boundary detection and computer graphics etc (see \cite{Bertrand94simple,KongRosen96topological}). This study was first considered in the 1970s by Rosenfeld \cite{Rosen79digital}. Since then, many researchers have tried to understand whether digital images have similar properties to the continuous analogues they represent. The main goal is to develop some kind of theory for digital images which is similar to the theory of topological spaces. However, due to the discrete and combinatorial nature of digital images, it is often difficult to get results that agree with their counterparts in classical topology. Many concepts in classical and algebraic topology have been developed successfully for digital images. Simplicial homology groups were not left behind as they were first studied in digital topology in \cite{Arslanetal08homology,Boxeretal11topological}. Later, digital cubical homology groups were also established by micmicking the cubical homology groups of topological spaces in algebraic topology \cite{JamilAli19digital,Staecker20digital}. Unlike in classical topology, the digital simplicial and cubical homology give two non-isomorphic homology theories associated to a digital image.

In the paper \cite{EgeKaraca13lefschetz}, Ege \& Karaca tried to give a digital analogue of the classical Lefschetz fixed point theory. However, there were errors in their main results, which were eventually retracted by the same authors with others in \cite{Boxeretal16digital}.
The aim of our paper is to give a correct approach to the topic, using the digital simplicial and cubical homology theories to define homotopy invariant Lefschetz numbers for any continuous functions on digital images. In the case of simplicial homology, we use the same definition of the Lefschetz number given in \cite{EgeKaraca13lefschetz}, but our fixed point theorem and homotopy invariance are necessarily different from the incorrect ones stated in that paper.

The rest of this manuscript is organized in the following manner: Section 2 gives a brief review of the basic background required in the subsequent sections. In Section 3, we recall the simplicial Lefschetz number and study its strong homotopic invariance property. Moreover, we present a $1$-approximate fixed point existence theorem. In Section 4, we introduce the cubical Lefschetz number and study its homotopic invariance property, we also establish the existence of an $n$-approximate fixed point theorem. In Section 5 we demonstrate a few more interesting properties of the Lefschetz numbers. We then end the paper with our concluding remarks.

\section{Preliminaries}
A \emph{digital image} is a pair $(X, \kappa),$ where $X$ is a set and $\kappa$ is a symmetric and antireflexive relation on $X$ called the \emph{adjacency relation}. A digital image $(X, \kappa)$ can be viewed as a graph for which $X$ is the vertex set and $\kappa$ determines the edge set. For all examples in this paper, $X$ will be a finite subset of $\mathbb Z^n$, and the adjacency will represent some sort of ``closeness" of the points in $\mathbb{Z}^n$. 

\subsection{Adjacencies and Homotopy}
We write $x \leftrightarrow_{\kappa} y$ to indicate that $x$ and $y$ are $\kappa$-adjacent and use the notation $x \Leftrightarrow_{\kappa} y$ to indicate that $x$ and $y$ are $\kappa$-adjacent or  equal. Or simply, we use $x \leftrightarrow y$ and $x \Leftrightarrow y$ whenever the adjacency $\kappa$ is understood or unnecessary to mention.

In this paper, we will generally use the following standard adjacencies on $\Z^n$: For $t \in \mathbb{N}$ with $1 \leq t \leq n$, any two distinct points $p = (p_1, p_2, \ldots , p_n)$ and $q = (q_1, q_2, \ldots , q_n)$ in $\mathbb{Z}^n$ are said to be $c_t$-adjacent if at most $t$ of their coordinates differ by $\pm 1,$ and all others coincide. Often, this adjacency is described in terms of how many points in $\mathbb Z^n$ are adjacent to a given point. For instance, when $n=1$, the $c_1$-adjacency is often called $2$-adjacency. In dimension 2, we have  $c_1$-adjacency and  $c_2$-adjacency as $4$-adjacency and $8$-adjacency respectively. 
%
%
\begin{definition}  \cite{Boxer94digitally}
	A ``digital interval" is defined as the set $\lbrack a, b \rbrack _{\mathbb{Z}} = \{n \in \mathbb{Z} \, | \, a \leq n \leq b\}$ together with the $2$-adjacency relation, where $a, b \in \mathbb{Z}$ such that $a \lneq b$.
\end{definition}
\begin{definition} \cite{Rosen86continuous,Boxer99classical}
	A function $f : X \longrightarrow Y$ between digital images $(X, \kappa_1)$ and $(Y, \kappa_2)$ is $(\kappa_1, \kappa_2)$-continuous if and only if for every $x, y \in X, f(x) \Leftrightarrow_{\kappa_2} f(y)$ whenever $x \leftrightarrow_{\kappa_1} y$.
\end{definition}
If $\kappa_1 = \kappa_2 = \kappa$, we say that a function is $\kappa$-continuous to abbreviate $(\kappa, \kappa)$-continuous. Whenever $\kappa_1$ and $\kappa_2$ are understood, we simply call the function continuous. 
\begin{definition} \cite{Khalimsky87motion}
	A digital $\kappa$-path in a digital image $(X,\kappa)$ is a $(2, \kappa)$-continuous function $\gamma : \lbrack 0, m \rbrack_{\mathbb{Z}} \longrightarrow X$. Further, $\gamma$ is called a digital $\kappa$-loop if $\gamma(0) = \gamma(m),$ and the point $p = \gamma(0)$ is called the base point of the loop $\gamma.$ Moreover, if $\gamma$ is a constant function, then $\gamma$ is called a trivial loop.
\end{definition}
\begin{definition}\label{defdhomo} \cite{Boxer99classical}
	Let $(X, \kappa_1)$ and $(Y, \kappa_2)$ be digital images. Suppose that $f, g : X \longrightarrow Y$ are $(\kappa_1, \kappa_2)$-continuous functions, and that there is a positive integer $m$ and a function $H : X \times \lbrack 0, m \rbrack_{\mathbb{Z}} \longrightarrow Y$ such that:
	\begin{itemize} 
		\item For all $x \in X, H(x, 0) = f(x)$ and $H(x, m) = g(x)$;
		\item For all $x \in X,$ the induced function $H_x : \lbrack 0, m \rbrack_{\mathbb{Z}} \longrightarrow Y$ defined by 
		\[H_x(t) = H(x, t), \mbox{ for all } t \in \lbrack 0, m \rbrack_{\mathbb{Z}} \]
		is $(c_1, \kappa_1)$-continuous. That is, $H_x(t)$ is a $\kappa$-path in $Y$;
		\item For all $t \in \lbrack 0, m \rbrack_{\mathbb{Z}},$ the induced function $H_t : X \longrightarrow Y$ defined by
		\[H_t(x) = H(x, t), \mbox{ for all } x \in X\]
		is $(\kappa_1, \kappa_2)$-continuous.
	\end{itemize}
	Then $H$ is a digital homotopy (or $\kappa$-homotopy) between $f$ and $g$. In this case, the functions $f$ and $g$ are said to be digitally homotopic (or $\kappa$-homotopic), which we denote by $f \simeq g.$
\end{definition}

The definition of homotopy can be expressed more succinctly as in classical topology by defining an adjacency relation on the product $X \times \lbrack 0,m \rbrack_\Z$. Unlike the classical situation in which there is a well-defined unique product topology, we have several natural choices for product adjacencies, in \cite{Boxer16generalized} called the \emph{normal product adjacencies}:
\begin{definition} \cite{Boxer16generalized}
	Let $(X_i, i)$ be digital images for each $i \in \{1, \ldots, n\}$. Then for some $u \in \{1, \ldots, n\},$ the normal product adjacency $NP_u(\kappa_1, \ldots, \kappa_n)$ is the adjacency relation on $\Pi^n_{i=1} X_i$ defined by: the p$(x_1, \ldots, x_n)$ and $(x'_1, \ldots, x'_n)$ are adjacent if and only if their coordinates are adjacent in at most $u$ positions, and equal in all other positions.
\end{definition}
When the underlying adjacencies are understood, we simply use $NP_u$ instead of $NP_u(\kappa_1, \ldots, \kappa_n).$ 

Boxer showed:
\begin{theorem}\label{hthm}\cite{Boxer16generalized}
	Let $(X, \kappa)$ and $(Y, \lambda)$ be digital images. Then $H : X \times \lbrack 0, k \rbrack_{\mathbb{Z}} \longrightarrow Y$ is a homotopy if and only if $H$ is $(NP_1(\kappa, c_1), \lambda)$-continuous.
\end{theorem}
In \cite{Staecker20digital}, the third author explored another homotopy definition, by simply using $NP_2$ in place of $NP_1.$ This imposes extra restrictions on the homotopy $H$. This stronger homotopy relation is also used in \cite{Luptonetal19fundamental}.
\begin{definition} \cite{Staecker20digital}
	Let $(X, \kappa)$ and $(Y, \lambda)$ be digital images. We say that $H^{*} : X \times \lbrack 0, k \rbrack_{\mathbb{Z}} \longrightarrow Y$ is a strong homotopy when $H^{*}$ is $(NP_2(\kappa, c_1), \lambda)$-continuous. If there is a strong homotopy $H^{*}$ between $f$ and $g$, we say $f$ and $g$ are strongly homotopic, and we write $f \simeq^{*} g.$ 
\end{definition}
Now, we present the following definitions which will be used later.
\begin{definition}\label{equitype} 
	Let $X$ and $Y$ be digital images. If $f : X \longrightarrow Y$ and $g : Y \longrightarrow X$ are continuous functions such that $g \circ f \simeq id_X$ and
	$f \circ g \simeq id_Y,$ where $id_X$ and $id_Y$ denote the identity mappings on $X$ and $Y.$ Then $f$ and $g$ are called homotopy equivalences. Moreover, $X$ and $Y$ are said to have the same homotopy type.
\end{definition}
\begin{definition} \label{strng_equitype}
	Let $X$ and $Y$ be digital images. If $f : X \longrightarrow Y$ and $g : Y \longrightarrow X$ are continuous functions such that $g \circ f \simeq^{*} id_X$ and
	$f \circ g \simeq^{*} id_Y,$ where $id_X$ and $id_Y$ denote the identity mappings on $X$ and $Y.$ Then $f$ and $g$ are called strong homotopy equivalences. Moreover, $X$ and $Y$ are  said to have the same strong homotopy type.
\end{definition}
Given any function $f:X\to X$, an element $x\in X$ is a \emph{fixed point} if $f(x)=x$. When $f(x) \Leftrightarrow x$ we say $x$ is an \emph{approximate fixed point}, and more generally if there is a path from $x$ to $f(x)$ of length $n$ or less, we say $x$ is a \emph{$n$-approximate fixed point}. Clearly a $1$-approximate fixed point is an approximate fixed point, and a $0$-approximate fixed point is a fixed point.

\subsection{Simplicial Homology}
Digital simplicial homology theory was studied initially by Arslan et al. in \cite{Arslanetal08homology}. This was later extended and published in English by Boxer et al. \cite{Boxeretal11topological}. We will review some basic definitions presented in \cite{Boxeretal11topological}.
\begin{definition} 
	Let $(X, \kappa)$ be a finite digital image and $q \geq 0$. Then a $q$-simplex is defined to be any set of $q+1$ mutually adjacent points of $(X, \kappa)$.
\end{definition} 
Let $x_0, \ldots , x_q,$ be some ordered list of mutually adjacent points and the associated ordered $q$-simplex be denoted by $\langle x_0, \ldots , x_q \rangle.$ The simplicial chain group $C_q(X)$ is defined as the abelian group generated by the set of all ordered $q$-simplices, where if $\rho : \{0, \ldots , q\} \longrightarrow \{0, \ldots , q\}$ is a permutation, then in $C_q(X)$ we identify
$\langle x_0, \ldots , x_q \rangle = (-1)^\rho \langle x_{\rho_0}, \ldots , x_{\rho_q} \rangle$
where $(-1)^{\rho} = 1$ when $\rho$ is an even permutation, and $(-1)^{\rho} = -1$ when $\rho$ is an odd permutation.

The boundary homomorphism $\partial_q : C_q(X) \longrightarrow C_{q-1}(X)$ is the homomorphism induced by defining:
\[\partial_q (\langle x_0, \ldots , x_q \rangle) = \sum^q_{i=0} (-1)^i \langle x_0, \ldots , \widehat{x_i} , \ldots x_q \rangle,\]
where $\widehat{x_i}$ indicates omission of the $x_i$ coordinate.

It can be verified that $\partial_{q-1} \circ \partial_q = 0,$ and so $(C_q(X), \partial_q)$ forms a chain complex, and the dimension $q$ homology group $H_q(X) = Z_q(X) / B_q(X)$ is defined to be the homology of this simplicial chain complex.
\begin{definition} \cite{Arslanetal08homology}
	Let $(X, \kappa)$ be a finite digital image. Then
	\begin{enumerate} 
		\item $Z^{\kappa}_q(X) = \Ker \partial_q$ is called the group of $q$-simplicial cycles;
		\item $B^{\kappa}_q(X) = \Ima \partial_{q+1}$ is called the group of $q$-simplicial boundaries;
		\item $H^{\kappa}_q(X) = Z^{\kappa}_q(X) / B^{\kappa}_q(X)$ is called the $q$th simplicial homology group.
	\end{enumerate}
\end{definition} 

\begin{definition}\cite{Boxeretal11topological} \label{defsm}
	Let $\phi : (X, \kappa) \longrightarrow (Y, \lambda)$ be a digitally continuous mapping. For $q \geq 0,$ we define a homomorphism $\phi_q : C_q^{\kappa}(X) \longrightarrow C_q^{\lambda}(Y)$ by 
	\[\phi_q(\langle p_0 , \ldots , p_q \rangle ) = \langle \phi(p_0), \ldots , \phi(p_q) \rangle,\]
	where the right side is interpreted as 0 if the set $\{\phi(p_0), \ldots , \phi(p_q)\}$ has fewer than $q+1$ points.
\end{definition}
As a consequence of Definition \ref{defsm}, we have the following lemma.
\begin{lemma}\cite{Boxeretal11topological}\label{lemchain}
	If $\phi : (X, \kappa) \longrightarrow (Y, \lambda)$ is a digitally continuous mapping then $\phi_q : C_q^{\kappa}(X) \longrightarrow C_q^{\lambda}(Y)$ is a chain map for each $q$. That is, $\phi_q \partial = \partial \phi_q.$
\end{lemma}

\subsection{Cubical Homology}
Cubical homology theory was introduced by Jamil and Ali \cite{JamilAli19digital} by imitating the classical cubical homology theory. However, prior to that, Karaca and Ege in \cite{KaracaEge12cubical} establish another type of cubical homology theory based on some classical constructions. Their focus on ``cubical sets" imposes the significant restriction that the digital image $X \subseteq \mathbb{Z}^n$ must use $c_1$-adjacency. This theory was modified and expanded in \cite{Staecker20digital}.

Now, we will recall some important definitions and basic results related to this theory based on the latter approach. 
\begin{definition} \cite{KaracaEge12cubical}
	An elementary interval is any set: $\lbrack a, a + 1\rbrack_{\mathbb{Z}} = \{a, a + 1\}$ or $\lbrack a, a\rbrack_{\mathbb{Z}} = \{a\}.$ An elementary interval of 1 point is called degenerate, and one of 2 points is called nondegenerate.
\end{definition}
\begin{definition} \cite{Staecker20digital}
	An elementary cube is any set: $\sigma = t_1 \times \ldots \times t_n,$ where each $t_i$ is an elementary interval. The number of nondegenerate factors is called the dimension of $\sigma$.
\end{definition}
Note that by elementary $q$-cube we mean an elementary cube of dimension $q$. Moreover, whenever $X \subset \mathbb{Z}^n$ is a digital image with $c_1$-adjacency, we can uniquely express $X$ as a union of elementary cubes.

Let $\bar C^{c_1}_q(X)$ be the free abelian group whose basis is the set of all $q$-cubes in $X.$ The cubical boundary operator is defined in terms of cubical face operators. i.e. For some elementary $q$-cube $\sigma = t_1 \times \ldots \times t_n,$ we define $q-1$-cubes $A_i$ and $B_i$ as:
\begin{align*}
A_i \sigma &= (t_1 \times \ldots \times t_{i-1} \times \min{t_i} \times t_{i+1} \times \ldots \times t_{n}), \\
B_i \sigma &= (t_1 \times \ldots \times t_{i-1} \times \max{t_i} \times t_{i+1} \times \ldots \times t_{n}).
\end{align*}
These $A_i$ and $B_i$ give the ``front face" and ``back face" of the cube in each of its $q$ dimensions. Note that when $t_i$ is a degenerate interval, we have $A_i \sigma = B_i \sigma.$ Whereas when $t_i$ is a nondegenerate interval, $A_i \sigma $ and $B_i \sigma$ are distinct.

Given an elementary $q$-cube $\sigma = t_1 \times \ldots \times t_n,$ let $(j_1, \ldots, j_q)$ be the sequence of indices for which $t_{j_i}$ is nondegenerate. We define the boundary operator $\partial_q : \bar C^{c_1}_q(X) \longrightarrow \bar C^{c_1}_{q-1} (X)$ on $q$-cubes by the following formula:
\[\partial_q(\sigma) = \sum^q_{i=1}(-1)^i(A_{j_i} \sigma - B_{j_i} \sigma).\]
It can be verified that $\partial_{q-1} \circ \partial_q = 0,$ and thus $(\bar C^{c_1}_q(X), \partial_q)$ is a chain complex.
\begin{definition} \cite{Staecker20digital}
	Let $(X, \kappa)$ be a finite digital image. Then
	\begin{itemize} 
		\item $\bar Z^{c_1}_q(X) = \Ker \partial_q$ is called the group of $q$th $c_1$-cubical cycles;
		\item $\bar B^{c_1}_q(X) = \Ima \partial_{q+1}$ is called the group of $q$th $c_1$-cubical boundaries;
		\item $\bar H^{c_1}_q(X) = \bar Z^{c_1}_q(X) / \bar B^{c_1}_q(X)$ is called the $q$th $c_1$-cubical homology group.
	\end{itemize}
\end{definition} 
\begin{theorem}\cite{Staecker20digital}\label{c1chainmap}
	Let $X \subset \mathbb{Z}^n$ and $X \subset \mathbb{Z}^m$ be digital images with $c_1$-adjacency and $n \leq 4,$ and let $f : X \longrightarrow Y$ be continuous. Then $f_{q} : \bar C^{c_1}_q(X) \longrightarrow \bar C^{c_1}_q(Y )$ is a
	chain map. That is, $f_q \partial = \partial f_q.$
\end{theorem}
The theorem above was proved by computer enumeration of all possible elementary $q$-cubes in low dimensions. A proof by hand would still be desirable. The restriction to $n\le 4$ was for performance reasons in the enumerations: \cite{Staecker20digital} conjectures that $f_\#$ is a chain map in all dimensions. 
\subsection{Traces in Homology}
In this section, we review a fact from homological algebra which is foundational for computing Lefschetz numbers. We will let $(C_q,\partial_q)$ be any chain complex. That is, each $C_q$ is a finitely generated abelian group, and each $\partial_q: C_q \to C_{q-1}$ is a homomorphism satisfying $\partial_q \circ \partial_{q+1} = 0$. A sequence of homomorphisms $f_q:C_q\to C_q$ is called a \emph{chain map} when $f_q\circ \partial = \partial \circ f_q$ for each $q$. In this case, $f_q$ induces a homomorphism $f_{*,q}:H_q \to H_q$, where $H_q$ are the homology groups of the chain complex.

Now, we will require a purely algebraic lemma:
\begin{lemma}\cite{Naber80topological}\label{lemnabor}
	Let $\psi : G \longrightarrow G$ be an endomorphism of a finitely generated Abelian group $G$ with a subgroup $H$ such that $\psi(H) \subseteq H$. Then \[\tr(\psi) = \tr(\hat{\psi}) + \tr(\psi| H).\]
\end{lemma}
The following is often referred to as the Hopf Trace Formula in the literature (see \cite{Brown71lefschetz}). For completeness we will give a proof here.
\begin{theorem}\label{hopf_trace}
	Let $(C_q,\partial_q)$ be a chain complex, and let $f_q:C_q \to C_q$ be a chain map and $f_{*,q}:H_q \to H_q$ be the homomorphism induced by $f_q$. In addition, let $B_{-1} = B_n = 0.$ Then 
	\[\sum_{q=0}^n (-1)^q \tr(f_q) = \sum_{q=0}^n (-1)^q \tr(f_{*,q}).\]
\end{theorem}
\begin{proof}
	As usual we write $Z_q = \ker \partial_q$ and $B_q = \Ima \partial_{q+1}$. Since $f_q$ is a chain map, we can observe that both $f_q|Z_q$ and $f_q|B_q$ are endomorphisms for any $q$. By Lemma \ref{lemnabor}, we have
	\[\tr(f_q) = \tr(\hat{f}_q) + \tr(f| Z_q)\] 
	where $\hat{f}_q: C_q / Z_q \longrightarrow C_q / Z_q$ is the induced endomorphism. We have, $B_{q-1} = \Ima \partial_{q}$ which is isomorphic to $C_q / Z_q$ and by using the fact that $f_q$ is a chain map and identifying $B_{q-1}$ and $C_q / Z_q$, we obtain $\tr(f_q) = \tr(f_{q-1}|B_{q-1})$ so that
	\begin{equation}\label{eq1}
	\tr(f_q) = \tr(f_{q-1}| B_{q-1}) + \tr(f_q| Z_q).
	\end{equation}
	Similarly, we have
	\begin{equation}\label{eq2}
	\tr(f_{*,q}) = \tr(f_q| Z_q) - \tr(f_q| B_q),
	\end{equation}
	since $B_q$ is a subgroup of $Z_q$ and invariant under the map $f_{q}| Z_{q}.$ Now, from (\ref{eq1}) and (\ref{eq2}), we have
	\begin{equation}\label{eq3}
	\tr(f_q) = \tr(f_{q-1}| B_{q-1}) + \tr(f_{*,q}) + \tr(f_q| B_q).
	\end{equation}
	Since dim $X = n$ and multiplying equation (\ref{eq3}) by $(-1)^q$ then summing over all $q$, we obtain
	\begin{equation}\label{eq4}
	\sum_{q=0}^n (-1)^q \tr(f_q) =  \sum_{q=0}^n (-1)^q \tr(f_{q-1}| B_{q-1}) + \sum_{q=0}^n (-1)^q \tr(f_{*,q}) + \sum_{q=0}^n (-1)^q \tr(f_q| B_q). 
	\end{equation}
	Since $B_{-1} = B_n = 0$, which further suggests that $\tr(f_q| B_q)$ occurs twice with opposite signs, equation (\ref{eq4}) becomes
	\[\sum_{q=0}^n (-1)^q \tr(f_q) = \sum_{q=0}^n (-1)^q \tr(f_{*,q})\]
	as required.
\end{proof}
The assumptions that $B_{-1} = B_n = 0$ will always hold in both the simplicial and $c_1$-cubical homology groups for a digital image, for appropriate choice of $n$.

\section{Simplicial Lefschetz Number}
In this section, we will define the simplicial Lefschetz number in the usual way as an alternating sum of traces. Using the simplicial homology theory, we obtain a ``simplicial Lefschetz number" $L(f)$.
\begin{definition}
	Let $X$ be a digital image and $n$ be the maximal dimension of any simplex in $X$. For a map $f : X \longrightarrow X,$ we define the simplicial Lefschetz number, denoted by $L(f)$ as follows:
	\[L(f) = \sum_{q=0}^n (-1)^q \tr(f_{*,q}),\]
	where $f_{*,q} : H_q^{\kappa}(X) \longrightarrow H_q^{\kappa}(X).$
\end{definition}
It is easy to compute the simplicial Lefschetz number of a constant.
\begin{example}\label{ex6}
	Let $(X, \kappa)$ be a finite digital image and $c : (X, \kappa) \longrightarrow (X, \kappa)$ be the constant mapping. Then  $\tr c_{*,q} = 1$ in dimension 0, and vanishes elsewhere, and so $L(c) = 1$.
\end{example}
Now, we will recall the definition of simplicial Euler characteristics for digital images. It has been studied for some time: it appears in \cite{Han07digital, Boxeretal11topological}, where it is simply called the \emph{digital Euler characteristic}. 
\begin{definition}\label{seuler}
	Let $X$ be a digital image and $n$ be the maximal dimension of any simplex in $X$. We denote by $\chi(X)$ the \emph{simplicial Euler characteristic} of $X$ as defined below.
	\[\chi(X) = \sum_{q=0}^n (-1)^q \dim H_q(X) = \sum_{q=0}^n (-1)^q \dim C_q(X).\]
\end{definition}
Note that the equality in Definition \ref{seuler} was proved in \cite{Boxeretal11topological}. However, for completeness we will highlight how it works in our situation. The above equality holds, since we are using Theorem \ref{hopf_trace} on the identity mapping which obviously is a chain map in all dimensions.
\begin{example}\label{expSEX}
	Let $(Y,c_1)$ be the digital image of Fig. \ref{fig2nohole} and $(Z,c_1)$ be the digital image of Fig. \ref{fig3nohole}. Then 
	$\chi(Y) = -1$ and $\chi(Z) = -2$.
\end{example}
\begin{example}\label{ex_id_euler}
	Let $(X, \kappa)$ be a finite digital image and $id: (X, \kappa) \longrightarrow (X, \kappa)$ be the identity mapping. Then $L(id) = \chi(X)$. This holds because the identity map will fix all the $q$-dimensional simplices in $X$. Thus, when viewed as a matrix, $id_{q}$ is the identity matrix. Hence for each $q$, the trace of the induced homomorphism $id_{q}$ equals the dimension of $C_q(X)$. Thus $L(id) = \chi(X)$.
\end{example}

As we will see later on, contrary to claims in \cite{EgeKaraca13lefschetz}, the simplicial Lefschetz number $L(f)$ is not invariant under homotopies of $f$. This is because $f\simeq g$ will not generally imply that $f_* = g_*$.
However, it is shown in \cite{Staecker20digital} that $f \simeq^* g$ implies $f_* = g_*$. Thus we obtain:
\begin{theorem}\label{shithm}
	Let $(X, \kappa)$ be a finite digital image and $f,g : (X, \kappa) \longrightarrow (X, \kappa)$ be strongly homotopic mappings. Then $L(f) = L(g)$.
\end{theorem}
We will demonstrate the simplicial Lefschetz number with a few examples.
\begin{figure}    
	\centering
	\resizebox{3.8cm}{2cm}{%
		\begin{tikzpicture}[scale=1.5]
		\begin{scope}[very thick, every node/.style={sloped,allow upside down}]
		\draw[step=1cm,gray,very thin] (0,0) grid (2,1);
		\draw [blue, thick] (0,1) -- node{\midarrow} node[left,rotate=90,blue]{$e_5$} (0,0);
		\draw [blue, thick] (0,0) -- node{\midarrow} node[below,blue]{$e_0$} (1,0);
		\draw [blue, thick] (1,0) -- node{\midarrow} node[below,blue]{$e_1$} (2,0);
		\node[red] () at (.5,.5) {$A_0$}; 
		\node[red] () at (1.5,.5) {$A_1$}; 
		\draw [blue, thick] (1,1) -- node{\midarrow} node[above,rotate=180,blue]{$e_4$} (0,1);
		\draw [blue, thick] (2,1) -- node{\midarrow} node[above,rotate=180,blue]{$e_3$} (1,1);
		\draw [blue, thick] (2,0) -- node{\midarrow} node[right,rotate=270,blue]{$e_2$} (2,1);
		\draw [blue, thick] (1,0) -- node{\midarrow} node[left,rotate=270,blue]{$e_6$} (1,1);
		
		
		
		\filldraw[black] (0,0) circle (1pt) node[left]{$y_0$};
		\filldraw[black] (0,1) circle (1pt) node[left]{$y_5$};
		\filldraw[black] (1,0) circle (1pt) node[below]{$y_1$};
		\filldraw[black] (2,0) circle (1pt) node[right]{$y_2$};
		\filldraw[black] (1,1) circle (1pt) node[above]{$y_4$};
		\filldraw[black] (2,1) circle (1pt) node[right]{$y_3$};	
		\end{scope}
		\end{tikzpicture}
	}
	\caption{{\bf A 2-dimensional digital image $(Y,4)$.} We allow $y_i$'s to represent the points and $e_i$'s to represent the adjacencies. The red arrows indicate the direction we adopt for the construction of the complex (see the examples below). $A_i$'s represent the 2-dimensional cubes (see examples below).}
	\label{fig2nohole}
\end{figure}
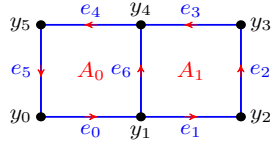
\begin{example}\label{ex_sim}
	Let $Y=\{y_0,\ldots,y_5\}$ with adjacency set $E= \{e_0,\ldots,e_6\}$ be the digital image of Fig. \ref{fig2nohole} and $f$ be the $180^{\circ}$ rotation map on $Y$. Then the simplicial Lefschetz number of $f$ is calculated as follows: By definition of $L(f)$ and Theorem \ref{hopf_trace}, we have 
	\[L(f) = \sum_{q=0}^n (-1)^q \tr(f_{q,*}) = \sum_{q=0}^n (-1)^q \tr(f_{q}).\]
	In dimensions 0 and 1, we have:
	$$
	f_0 =
	\begin{pmatrix}
	0&0&0&1&0&0\\
	0&0&0&0&1&0\\
	0&0&0&0&0&1\\
	1&0&0&0&0&0\\
	0&1&0&0&0&0\\
	0&0&1&0&0&0
	\end{pmatrix},
	\quad
	f_1 =
	\begin{pmatrix}
	0&0&0&1&0&0&0\\
	0&0&0&0&1&0&0\\
	0&0&0&0&0&1&0\\
	1&0&0&0&0&0&0\\
	0&1&0&0&0&0&0\\
	0&0&1&0&0&0&0\\
	0&0&0&0&0&0&-1
	\end{pmatrix}
	$$
	Hence $\tr(f_0) = 0$ and $\tr(f_1) = -1$. Since there are no simplices of dimension $2$ and above, then their simplicial homology is zero. Therefore, 
	\[ L(f)= 1(0)+(-1)(-1)=1. \]
\end{example} 
Note that all we care about are the traces of the matrices, not the matrices themselves. We can observe that $\tr(f_0)$ counts the number of fixed 0-simplices (vertices), $tr(f_1)$ gives an oriented count of the number of fixed 1-simplices (edges), $\tr(f_2)$ is an oriented count of the number of fixed 2-simplices, etc.
\begin{example}
	Let $Z=\{z_0,\ldots,z_7\}$ with adjacency set $E= \{e_0,\ldots,e_9\}$ be the digital image of Fig. \ref{fig3nohole} and $f$ be the $180^{\circ}$ rotation map on $Z$. Then the simplicial Lefschetz number of $f$ is calculated as follows: As before, we have
	\[L(f) = \sum_{q=0}^n (-1)^q \tr(f_q).\]
	For $0$-simplices and $1$-simplices, we have $\tr(f_0) = 0$ and $\tr(f_1) = 0$, because there are no fixed points (vertices) and no fixed edges respectively.
	Since there are no simplices of dimension 2 and higher, all other traces are zero. Therefore, 
	\[ L(f)= 1(0)+(-1)(0)=0. \]
\end{example} 
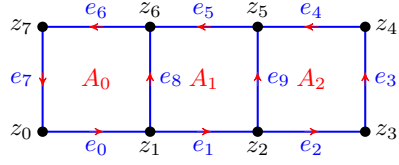
\begin{figure}    
	\centering
	\resizebox{5.5cm}{2.25cm}{%
		\begin{tikzpicture}[scale=1.5]
		\begin{scope}[very thick, every node/.style={sloped,allow upside down}]
		\draw[step=1cm,gray,very thin] (0,0) grid (3,1);
		\draw [blue, thick] (0,0) -- node{\midarrow} node[below,blue]{$e_0$} (1,0);
		\draw [blue, thick] (1,0) -- node{\midarrow} node[below,blue]{$e_1$} (2,0);
		\draw [blue, thick] (2,0) -- node{\midarrow} node[below,blue]{$e_2$} (3,0);
		\draw [blue, thick] (0,1) -- node{\midarrow} node[left,rotate=90,blue]{$e_7$} (0,0);
		\draw [blue, thick] (3,1) -- node{\midarrow} node[above, rotate=180,blue]{$e_4$} (2,1);
		\draw [blue, thick] (2,1) -- node{\midarrow} node[above,rotate=180,blue]{$e_5$} (1,1);
		\draw [blue, thick] (1,1) -- node{\midarrow} node[above,rotate=180,blue]{$e_6$} (0,1);
		\draw [blue, thick] (3,0) -- node{\midarrow} node[right,rotate=270,blue]{$e_3$} (3,1);
		\draw [blue, thick] (2,0) -- node{\midarrow} node[right,rotate=270,blue]{$e_9$} (2,1);
		\draw [blue, thick] (1,0) -- node{\midarrow} node[right,rotate=270,blue]{$e_8$} (1,1);
		\node[red] () at (.5,.5) {$A_0$}; 
		\node[red] () at (1.5,.5) {$A_1$};
		\node[red] () at (2.5,.5) {$A_2$}; 
		\filldraw[black] (0,0) circle (1pt) node[left]{$z_0$};
		\filldraw[black] (0,1) circle (1pt) node[left]{$z_7$};
		\filldraw[black] (1,0) circle (1pt) node[below]{$z_1$};
		\filldraw[black] (2,0) circle (1pt) node[below]{$z_2$};
		\filldraw[black] (1,1) circle (1pt) node[above]{$z_6$};
		\filldraw[black] (3,1) circle (1pt) node[right]{$z_4$};
		\filldraw[black] (3,0) circle (1pt) node[right]{$z_3$};
		\filldraw[black] (2,1) circle (1pt) node[above]{$z_5$};	
		\end{scope}
		\end{tikzpicture}
	}
	\caption{{\bf A 2-dimensional digital image $(Z,4)$.} We allow $z_i$'s to represent the points and $e_i$'s to represent the adjacencies. The red arrows indicate the direction we adopt for the construction of the complex (see the examples below). $A_i$'s represent the 2-dimensional cubes (see examples below).}
	\label{fig3nohole}
\end{figure}
\begin{example}
	Let $X=\{x_0,\ldots,x_{12}\}$ with adjacency set $E= \{e_0,\ldots,e_{13}\}$ be the digital image of Fig. \ref{fig2hole} and $f$ be the $180^{\circ}$ rotation map on $X$. 
	
	In this case there is one fixed vertex, no fixed edge, and no higher dimensional simplices. Thus we have $L(f)=1$.
\end{example} 
Of course the goal in defining a Lefschetz number is to obtain some sort of fixed point theorem. In our case when $L(f)$ is nonzero, we show that there is some simplex which maps to itself by $f$. This does not necessarily mean that there will be a fixed point, but it does imply that there is an approximate fixed point (that is, some $x$ with $f(x)$ adjacent or equal to $x$).
\begin{theorem}\label{sthm1}
	Let $(X,\kappa)$ be a finite digital image and $f:(X,\kappa) \longrightarrow (X,\kappa)$ be a self-map with $L(f)\neq 0$. Then there is some simplex $\sigma \subseteq X$ with $f(\sigma) = \sigma$.
\end{theorem}
\begin{proof}
	Since $L(f)\neq 0$ it implies that $\sum_{q=0}^n (-1)^q \tr(f_q) \neq 0$, where $f_q :C_q(X)\longrightarrow C_q(X)$ is the induced map on the simplicial chain groups. This further implies that there exist some dimension $q$ and some simplex $\sigma \subseteq X$ such that $f_q(\sigma) = \pm \sigma$, and so $f(\sigma)=\sigma$.
\end{proof}
\begin{corollary}\label{fixedorapprox}
	Let $(X, \kappa)$ be a finite digital image and $f : (X, \kappa) \longrightarrow (X, \kappa)$ be a self-map with $L(f) \not= 0.$ 
	Then either $f$ has a fixed point or $f$ has at least $2$ approximate fixed points.
\end{corollary}
\begin{proof}
	From Theorem \ref{sthm1}, if the fixed simplex is of dimension $0$, then $f$ has a fixed point. While if the fixed simplex is of higher dimension, then we have two cases; either some point is fixed by $f$, or all points of the simplex are approximate fixed points, and there will be at least $2$ of them.
\end{proof}
Now, as a consequence of Corollary \ref{fixedorapprox} and Theorem \ref{shithm} we have:
\begin{corollary}\label{sthm}
	Let $(X, \kappa)$ be a finite digital image and $f : (X, \kappa) \longrightarrow (X, \kappa)$ be a self-map with $L(f) \not= 0.$ 
	Then any map strong homotopic to $f$ has a fixed point or at least two approximate fixed points.
\end{corollary}
%
\begin{figure}    
	\centering
	\resizebox{6cm}{3cm}{%
		\begin{tikzpicture}[scale=1.5]
		\begin{scope}[very thick, every node/.style={sloped,allow upside down}]
		\draw[step=1cm,gray,very thin] (0,0) grid (4,2);
		\draw [blue, thick] (0,1) -- node{\midarrow} node[left,rotate=90,blue]{$e_{11}$} (0,0);
		\draw [blue, thick] (0,2) -- node{\midarrow} node[left,rotate=90,blue]{$e_{10}$} (0,1);
		\draw [blue, thick] (0,0) -- node{\midarrow} node[above,blue]{$e_0$} (1,0);
		\draw [blue, thick] (1,0) -- node{\midarrow} node[above,blue]{$e_1$} (2,0);
		\draw [blue, thick] (2,0) -- node{\midarrow} node[above,blue]{$e_2$} (3,0);
		\draw [blue, thick] (3,0) -- node{\midarrow} node[above,blue]{$e_3$} (4,0);	
		\draw [blue, thick] (4,2) -- node{\midarrow} node[below,rotate=180,blue]{$e_6$} (3,2);
		\draw [blue, thick] (3,2) -- node{\midarrow} node[below,rotate=180,blue]{$e_7$} (2,2);
		\draw [blue, thick] (2,2) -- node{\midarrow} node[below,rotate=180,blue]{$e_8$} (1,2);
		\draw [blue, thick] (1,2) -- node{\midarrow} node[below,rotate=180,blue]{$e_9$} (0,2);
		\draw [blue, thick] (4,1) -- node{\midarrow} node[right,rotate=270,blue]{$e_5$} (4,2);
		\draw [blue, thick] (4,0) -- node{\midarrow} node[right,rotate=270,blue]{$e_4$} (4,1);
		\draw [blue, thick] (2,1) -- node{\midarrow} node[left,rotate=270,blue]{$e_{13}$} (2,2);
		\draw [blue, thick] (2,0) -- node{\midarrow} node[left,rotate=270,blue]{$e_{12}$} (2,1);
		\filldraw[black] (0,0) circle (1.25pt) node[below]{$x_{0}$};
		\filldraw[black] (0,1) circle (1.25pt) node[left]{$x_{11}$};
		\filldraw[black] (0,2) circle (1.25pt) node[above]{$x_{10}$};
		\filldraw[black] (1,0) circle (1.25pt) node[below]{$x_1$};
		\filldraw[black] (2,0) circle (1.25pt) node[below]{$x_2$};
		\filldraw[black] (3,0) circle (1.25pt) node[below]{$x_3$};
		\filldraw[black] (4,0) circle (1.5pt) node[below]{$x_4$};
		\filldraw[black] (1,2) circle (1.25pt) node[above]{$x_9$};
		\filldraw[black] (2,2) circle (1.25pt) node[above]{$x_8$};
		\filldraw[black] (3,2) circle (1.25pt) node[above]{$x_7$};
		\filldraw[black] (4,2) circle (1.25pt) node[above]{$x_6$};
		\filldraw[black] (2,1) circle (1.25pt) node[right]{$x_{12}$};
		\filldraw[black] (4,1) circle (1.25pt) node[right]{$x_5$};
		\end{scope}
		\end{tikzpicture}
	}
	\caption{{{\bf A 2-dimensional digital image $(X,4)$.} We allow $x_i$'s to represent the points and $e_i$'s to represent the adjacencies. The red arrows indicate the direction we adopt for the construction of the complex (see the examples below). $A_i$'s represent the 2-dimensional cubes (see examples below).}}
	\label{fig2hole}
\end{figure}
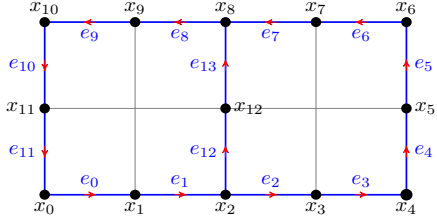
\begin{remark}
	Example \ref{ex_sim} demonstrates that $L(f)\neq 0$ will not in general imply the existence of a fixed point.
\end{remark}
\begin{theorem}\label{shafpthmid}
	Let $f$ be strong homotopic to identity and the simplicial Euler characteristic $\chi(X)$ be nonzero, 
	then $f$ has a fixed point or at least two approximate fixed points.
\end{theorem}
\begin{proof}
	Since $f$ is strong homotopic to identity, we have $L(f) = L(id) = \chi(X)$. Now, by Corollary \ref{fixedorapprox} and the fact that $\chi(X)\neq 0$, 
	we conclude that $f$ has a fixed point or at least two approximate fixed points.
\end{proof}
\begin{theorem}\label{strngcontrb}
	Let $X$ be strongly contractible digital image, 
	then any function $f$ on $X$ 
	has a fixed point or at least two approximate fixed points.
\end{theorem}
\begin{proof}
	Let $f$ be a map on $X$. Since $X$ is strongly contractible, then $f$ is strong homotopic to a constant map $c$. By Theorem \ref{shithm} and Example \ref{ex6}, we have $L(f) = L(c) = 1$ which implies that $L(f) \not= 0.$ Hence, by Corollary \ref{fixedorapprox} 
	we conclude that $f$ has a fixed point or at least two approximate fixed points.
\end{proof}

\section{Cubical Lefschetz Number}
In this section, we define a ``cubical Lefschetz number'' by using $c_1$-cubical homology in place of simplicial homology. We will illustrate with some examples that the cubical Lefschetz number $\bar L(f)$ can be different from the simplicial Lefschetz number $L(f)$ defined in the preceding section.

As was earlier mentioned in Section 2.3, cubical homology was defined in \cite{JamilAli19digital} for an image of any adjacency $\kappa$, where the cubes are $(c_1,\kappa)$-continuous functions $\lbrack 0,1\rbrack_\Z^n \longrightarrow X.$ The $c_1$-cubical homology defined in \cite{Staecker20digital} is more restrictive in that it requires $c_1$ adjacency to be used in $X$, but it is much easier to compute than the theory of \cite{JamilAli19digital}.

Since we are going to be using the $c_1$-cubical homology theory. We will denote the ``cubical Lefschetz number" as $\bar L(f)$.
\begin{definition}
	Let $X \subset \mathbb{Z}^n$ be a digital image with $c_1$-adjacency and $n \leq 4$. For a map $f : (X, c_1) \longrightarrow (X, c_1),$ we define the cubical Lefschetz number, denoted by $\bar L(f)$, as follows:
	\[\bar L(f) = \sum_{q=0}^n (-1)^q \tr(\bar f_{*,q}),\]
	where $\bar f_{*,q} : \bar H_q(X) \longrightarrow \bar H_q(X).$
\end{definition}
Computing the above number is made easier by using $\bar C_q(X)$ as guaranteed by Theorem \ref{hopf_trace}. However, for us to use Theorem \ref{hopf_trace}, we need $\bar f_q$ to be a chain map. This restricts us to only when $n \leq 4,$ since the induced homomorphism $\bar f_q$ is a chain map only up to dimension $4$, as was proved in  \cite{Staecker20digital}.

Similar to the simplicial section, we can easily compute the cubical Lefschetz number of a constant.
\begin{example}\label{ex7}
	Let $(X, c_1)$ be a finite digital image and $c : (X, c_1) \longrightarrow (X, c_1)$ be the constant mapping. Then $\bar L(c) = 1$. Because, for a constant map $\tr \bar f_{*,q} = 1$ in dimension 0, and vanishes elsewhere.
\end{example}
As was the case in the preceding section, the goals of defining the cubical Lefschetz number are: to obtain a homotopy invariance property and to establish some sort of fixed point theorem. We will see that the homotopy invariance property in the cubical setting is stronger than the simplicial setting, but the fixed point result is weaker.
\begin{theorem}\label{homoinvac1}\cite{Staecker20digital}
	Let $X \subset \mathbb{Z}^n$ and $Y \subset \mathbb{Z}^m$ be digital images both with $c_1$-adjacency. Let $n \leq 3$ and $f, g : X \longrightarrow Y$ be homotopic mappings. Then for each $q$, the induced homomorphisms $\bar f_*, \bar g_* :
	\bar H_q (X) \longrightarrow \bar H_q (Y)$ are equal.
\end{theorem}
Theorem \ref{homoinvac1} consequently, gives us the following:
\begin{theorem}\label{bhithm}
	Let $X \subset \mathbb{Z}^n$ be a digital image with $c_1$-adjacency and $n \leq 3$, and let $f,g : (X, c_1) \longrightarrow (X, c_1)$ be homotopic mappings. Then $\bar L(f) = \bar L(g)$.
\end{theorem}
\begin{proof}
	By Theorem \ref{homoinvac1}, for any two homotopic mappings their induced homomorphisms are equal. This implies that the alternating sum of their traces will also be equal. Hence $\bar L(f) = \bar L(g)$ as required.
\end{proof}
From Theorem \ref{hopf_trace}, we have
\[\bar L(f) = \sum_{q=0}^n (-1)^q \tr(\bar f_q) = \sum_{q=0}^n (-1)^q \tr(\bar f_{*,q}).\]
The above relation guarantees that we can now use the cubical chain groups to calculate the cubical Lefschetz number, which is much easier to do than using the cubical homology groups. \newline

Now, we will introduce the definition of cubical Euler characteristics for digital images. More precisely, we give the following definition.
\begin{definition}\label{ceuler}
	Let $X$ be a digital image and $n$ be the maximal dimension of any cube in $X$. We define the \emph{cubical Euler characteristic} of $X$ as:
	\[\bar \chi(X) = \sum_{q=0}^n (-1)^q \dim \bar H_q(X) = \sum_{q=0}^n (-1)^q \dim \bar C_q(X).\]
\end{definition}
The above equality holds, since we are using Theorem \ref{hopf_trace} on the identity mapping which obviously is a chain map in all dimensions. Therefore, the earlier requirement that $n \leq 4$ does not arise.
\begin{example}\label{expCEX}
	Let $(Y,c_1)$ be the digital image of Fig. \ref{fig2nohole} and $(Z,c_1)$ be the digital image of Fig. \ref{fig3nohole}. Then $\bar \chi(Y) = \bar \chi(Z) = 1$, while $\chi(Y) = -1$ and $\chi(Z)=-2$. Thus the cubical and simplicial Euler characteristics can differ.
\end{example}
\begin{example}\label{ex_id_ceuler}
	Let $(X, c_1)$ be a finite digital image and $id : (X, c_1) \longrightarrow (X, c_1)$ be the identity mapping. Then, $\bar L(id) = \bar \chi(X)$.
\end{example}
Although as highlighted earlier, we only care about the traces of the matrices not the matrices themselves, in the following example we will use the matrices method for the reader to see what happens with the cubical homology.
\begin{example}
	Let $(Y,c_1)$ be the digital image of Fig. \ref{fig2nohole}, and $f$ be the $180^{\circ}$ rotation map on $Y$. Then the cubical Lefschetz number of $f$ is calculated as follows: By definition of $\bar L(f)$ and Theorem \ref{hopf_trace}, we have 
	\[\bar L(f) = \sum_{q=0}^n (-1)^q \tr(\bar f_{q,*}) = \sum_{q=0}^n (-1)^q \tr(\bar f_q). \]
	We have:
	$$	\bar f_0 = \begin{pmatrix}
	0&0&0&1&0&0\\
	0&0&0&0&1&0\\
	0&0&0&0&0&1\\
	1&0&0&0&0&0\\
	0&1&0&0&0&0\\
	0&0&1&0&0&0
	\end{pmatrix}, \quad
	\bar f_1 =
	\begin{pmatrix}
	0&0&0&1&0&0&0\\
	0&0&0&0&1&0&0\\
	0&0&0&0&0&1&0\\
	1&0&0&0&0&0&0\\
	0&1&0&0&0&0&0\\
	0&0&1&0&0&0&0\\
	0&0&0&0&0&0&-1
	\end{pmatrix}, \quad
	\bar f_2 =
	\begin{pmatrix}
	0&1\\
	1&0
	\end{pmatrix}	$$
	Hence $\tr(\bar f_0) = 0$, $\tr(\bar f_1) = -1$, and $\tr(\bar f_2) = 0$. Therefore, 
	\[ \bar L(f) = 1(0)+(-1)(-1)+1(0)=1. \]
\end{example} 
Note that similar to the simplicial homology, we can observe that $\tr( \bar f_0)$ counts the number of fixed 0-cubes (vertices), $\tr(\bar f_1)$ is an oriented count of the number of fixed 1-cubes (edges), $\tr(\bar f_2)$ is an oriented count of the number of fixed 2-cubes, etc.
\begin{example}\label{exp_lbar}
	Let $(Z,c_1)$ be the digital image of Fig. \ref{fig3nohole}, and $f$ be the $180^{\circ}$ rotation map on $Z$. 
	%
	For $0$-cubes, we have $\tr(\bar f_0) = 0$ because there is no fixed vertex ($0$-cube). For $1$-cubes, we obtain $\tr(\bar f_1) = 0$ due to the absence of any fixed $1$-cube.
	For $2$-cubes, we have $\tr(\bar f_2) = 1$ as the function $f$ maps the $2$-cube $A_1$ onto itself with the same orientation. Now, as there are no higher dimensional cubes their $c_1$-cubical homology is zero.
	Hence, 
	\[ \bar L(f)=1(0)+(-1)(0)+1(1)=1. \]
\end{example} 

\begin{example}
	Let $(X,c_1)$ be the digital image of Fig. \ref{fig2hole}, and let $f$ be the $180^{\circ}$ rotation map on $X$. 
	For $0$-cubes, we have $\tr(\bar f_0) = 1$ due to the existence of a fixed vertex ($0$-cube). 
	For $1$-cubes, $\tr(\bar f_1) = 0$ because there is no fixed $1$-cube.
	Since there are no $2$-cubes and higher, their $c_1$-cubical homology is zero. Therefore, 
	\[ \bar L(f)= 1(1)+(-1)(0)=1. \]
\end{example}
For the existence of fixed points, when $\bar L(f)$ is nonzero, we will show that $f$ maps some elementary cube to itself. But this does not necessarily guarantee that we will have a fixed point or even an approximate fixed point. For example, an elementary $2$-cube can map to itself by a $180^{\circ}$ rotation, having no fixed or approximate fixed points. However, since a $q$-cube has diameter $q$ (where distance is measured along $c_1$-adjacencies), we see that a fixed $q$-cube must contain a $q$-approximate fixed point.
\begin{theorem}\label{cthm1}
	Let $(X,c_1)\subset \Z^n$ be a finite digital image with $n \leq 4$ and $f:(X,c_1) \longrightarrow (X,c_1)$ be a self map with $\bar L(f)\neq 0$. Then there is some cube $\sigma \subseteq X$ with $f(\sigma) = \sigma$.
\end{theorem}
\begin{proof}
	Since $\bar L(f)\neq 0$ it implies that $\sum_{q=0}^n (-1)^q \tr(\bar f_q) \neq 0$. This further implies that there exists some cube $\sigma \subseteq X$ such that $\bar f_q(\sigma) = \pm \sigma$, which means that $f(\sigma) = \sigma$.
\end{proof}
\begin{corollary}\label{cubicalfixedorapprox}
	Let $(X,c_1) \subset \Z^n$ be a finite digital image with $n \leq 4$ and $f : (X, c_1) \longrightarrow (X, c_1)$ be a self-map with $\bar L(f) \not= 0.$ 
	Then either $f$ has a fixed point or $f$ has at least $2$ $n$-approximate fixed points.
\end{corollary}
\begin{proof}
	From Theorem \ref{cthm1}, if the fixed cube is of dimension $0$, then $f$ has a fixed point. While if the fixed cube is of higher dimension, then we have two cases; either some point is fixed by $f$, or all points of the cube are $n$-approximate fixed points, and there are at least $2$ of them.
\end{proof}
Now, as a consequence of Theorem \ref{cthm1} and \ref{bhithm}, we have:
\begin{corollary}\label{cthm}
	Let $(X, c_1) \subset \Z^n$ be a finite digital image with $n \leq 4$ and $f : (X, c_1) \longrightarrow (X, c_1)$ be a self-map with $\bar L(f) \not= 0.$ 
	Then any function homotopic to $f$ has a fixed point or at least two $n$-approximate fixed points.
\end{corollary}
\begin{example}
	Let $I^n \subset \Z^n$ be a digital image with the $c_1$-adjacency, where $I=\{0,1\}$ and $f$ be the antipodal map on $I^n$. That is, $f((t_1, \ldots, t_n)) = (1-t_1, \ldots, 1-t_n)$ for every $(t_1, \ldots, t_n) \in I^n$. Then every point of $X$ is an $n$-approximate fixed point. In fact, $f$ can never have $k$-approximate fixed point for any $k < n.$ This demonstrates that it is not always possible to improve the result in Corollary \ref{cthm} to achieve "closer" approximate fixed points.
\end{example}
\begin{theorem}\label{chafpthmc}
	Let $X \subset \Z^n$, let $f$ be homotopic to identity and the cubical Euler characteristic $\bar \chi(X)$ be nonzero. 
	Then $f$ has a fixed point or at least two $n$-approximate fixed points.
\end{theorem}
\begin{proof}
	Since $f$ is homotopic to identity, we have $\bar L(f) = \bar L(id) = \bar \chi(X) \neq 0$. The conclusion follows from Corollary \ref{cubicalfixedorapprox}.
\end{proof}
\begin{theorem}\label{contrb}
	If $X \subset \Z^n$ is contractible digital image, 
	then any function $f$ on $X$ has a fixed point or at least two $n$-approximate fixed points.
\end{theorem}
\begin{proof}
	Let $f$ be a map on $X$. Since $X$ is contractible, then $f$ is homotopic to a constant map $c$. By Theorem \ref{bhithm} and Example \ref{ex7}, we have $\bar L(f) = \bar L(c) = 1$ which implies that $\bar L(f) \not= 0.$ Hence, by Corollary \ref{cthm} we conclude that $f$ has a fixed point or at least two $n$-approximate fixed points.
\end{proof}

Now we will discuss the values of the Lefschetz number when $X$ is a digital cycle of points. There is a simple classification of all continuous self-maps on digital cycles up to homotopy:

\begin{theorem}\cite{BoxerSta19fixed}\label{mapschar}
	Let $C_n = \{x_0,\dots, x_{n-1}\}$ be the digital cycle of $n$ points. Define the flip map $t$ on $C_n$ as $t(x_i)=x_{-i}$ for all $x_i \in C_n$, where subscripts are read modulo $n$. Then any self-map on $C_n$ is homotopic to one of: a constant map, the identity map, or the flip map $t$.
\end{theorem}
\begin{proposition}\label{prop}
	Let $C_n$ be a digital cycle of $n$ points. If $f$ is any continuous self-map on $C_n$, then $L(f) = \bar L(f) \in \{0,1,2\}$, when $n \neq 4$, and $L(f) \in \{0,1,2\}$ and $\bar L(f) =1$ for $n = 4$. 
\end{proposition}
\begin{proof}
	We are going to prove this in three cases as follows:
	\begin{itemize}
		\item Case 1: when $n > 4$;
		\item Case 2: when $n < 4$;
		\item Case 3: when $n = 4.$
	\end{itemize}
	\textbf{Case 1:} For $n>4$ there are no $2$ or higher dimensional simplices or cubes in $C_n.$ This implies that the computation for $L$ and $\bar L$ rely only on dimensions $0$ and $1$ simplices and cubes respectively. Therefore, since in dimensions $0$ and $1$ the traces are equal, we have $L(f) = \bar L(f)$. More precisely, in dimension $0$ both traces count the number of fixed points, and similarly in dimension $1$ they both count the number of fixed edges (with orientation). 
	
	Now, it remains only to show that the value of $L(f) = \bar L(f)$ is in fact in the set $\{0,1,2\}.$ So it is enough show that for any self map $f$ on $C_n$, $\bar L(f) \in \{0,1,2\}.$
	From Theorem \ref{mapschar}, we know that $f$ is homotopic to either the identity, a constant, or the flip map. We will consider them one after the other. If $f$ is homotopic to the identity map then by Theorem \ref{bhithm} and Example \ref{ex_id_ceuler}, $\bar L(f) = \bar \chi_{c_1}(C_n) = 0.$  If $f$ is homotopic to the constant map then by Theorem \ref{bhithm} and Example \ref{ex7}, $\bar L(f) = 1.$ Lastly, if $f$ is homotopic to the flip map $t$ then by Theorem \ref{bhithm}, we have $\bar L(f) = \bar L(t)$. Now, we conclude that if $n$ is even, we have $\bar L(f) = 2$ since $t$ has $2$ fixed points ($0$-simplices) and no fixed edges ($1$-simplices). If $n$ is odd, we have $\bar L(f) = 2,$ since $t$ has one fixed point ($0$-simplex) and one fixed edge ($1$-simplex) in the reverse orientation. Therefore, $L(f) = \bar L(f) \in \{0,1,2\}$ for any map $f$ on $C_n$ with $n>4.$
	
	\textbf{Case 2:} When $n<4$, this is easy as $C_n$ is both contractible and strong contractible, hence by Theorem \ref{shithm},\ref{bhithm} and Example \ref{ex6}, \ref{ex7}, $L(f) = \bar L(f) = 1$.
	
	\textbf{Case 3:} For $n=4$, we have a special case, where $C_4$ is contractible but not strong contractible.
	Now, we will compute both $L(f)$ and  $\bar L(f)$ case by case, with the help of the characterization of maps on $C_4$ by Theorem \ref{mapschar}. For $L(f)$; if $f$ is a constant map then $L(f)=1$ by Theorem \ref{shithm} and Example \ref{ex6}. If $f$ is a rotation map, then it has no fixed points or edges, so $L(f)=0.$ If $f$ is the identity map, then computing the traces directly gives $L(f)=0,$ which is also in line with Example \ref{ex_id_euler}. If $f$ is the flip map $t$, we then obtain $L(f) = 2$. This is because, there are only $2$ fixed points ($0$-simplices) and no fixed edges ($1$-simplices). Similarly, if $f$ is composition of the flip map $t$ and a rotation map, we have $L(f) = 2$. Therefore, $L(f) \in \{0,1,2\}$ for any map $f$ on $C_4.$ Now, for $\bar L(f)$; we can easily see that $\bar L(f) = \bar L(c)=1$ from Theorem \ref{bhithm} and Example \ref{ex7}. 
\end{proof}
Note that, the argument in the proof above is in agreement with the classical method of calculating $L(f)$ on a circle, which has the formula as: $L(f) = 1-d,$ where $d$ is the degree of the mapping. Since in a digital cycle the maps; identity, constant and flip have $1,0$ and $-1$ degree respectively.

\section{Properties of Digital Lefschetz Numbers}

\subsection{Digital Lefschetz Numbers and Number of Approximate Fixed Points}
In classical topology, the classical Lefschetz fixed point theorem says: ``If $L(f) \neq 0$ then the map $f$ has a fixed point.'' However, the specific values of $L(f)$ do not give us much more information apart from that. For instance, if $L(f) = 2$, then this does not mean that there will be $2$ fixed points: there could be a single fixed point with fixed point index 2.

The situation is different, however, in the digital setting. When $f:X\to Y$ is continuous, the entries of the matrices $f_q$ and $\bar f_q$ are always $0, 1,$ or $-1.$ For example if $s$ is a $q$-simplex, then the column of $f_q$ corresponding to $s$ will have an entry of $\pm 1$ in the position indicating the $q$-simplex $f(s)$, or will be 0 if $f(s)$ is of dimension smaller than $q$.

In the classical setting, the matrix can be more interesting because an edge can stretch and wrap across several others or even itself multiple times, which could give matrix entries other than 0 or $\pm 1$. However, this kind of scenario is not possible in digital topological setting since a digitally continuous map cannot stretch simplices or cubes.

Therefore, for instance if we have $|L(f)| = k$ with $k \neq 0$, it will imply that there are at least $k$ fixed simplices which make nonzero contribution to the computation of $L(f)$ and hence add up to $k$. Similarly to $\bar L(f).$ Thus, we obtain:
\begin{theorem}
	If $|L(f)| = k$, then there are at least $k$ fixed simplices, and hence at least $k$ distinct approximate fixed points.
\end{theorem}
\begin{example}
	Let $X$ be the digital image in Figure $\ref{robot}$ and $id$ be the identity map on $X$, then $L(id)=\chi(X)=-2$. Now, from the above theorem and the strong homotopy invariance property of $L(f)$, any map strong homotopic to the identity must have at least $2$ distinct approximate fixed points.
	
	We also note that $\bar L(id) = \bar \chi(X) = 0$, and that the identity map is in fact homotopic to a map with no 2-approximate fixed points. We may change $id$ by homotopy by contracting the ``arms'' and ``legs'' so that all of its values occur in the long rectangle at right in Figure \ref{robot}. Then we may rotate these values so that no point is 2-approximate fixed.
\end{example}

Since any self-map has at least $|L(f)|$ distinct approximate fixed points, we obtain:
\begin{corollary}\label{cor_supbd}
	Let $L(f)$ be the simplicial Lefschetz number of a continuous self-map $f$ on a digital image $X.$ Then, $|L(f)| \leq \# X.$
\end{corollary}

By the same reasoning as above we obtain similar results for the cubical Lefschetz number.
\begin{theorem}\label{no_nafpt}
	If $|\bar L(f)| = k$, then there are at least $k$ fixed cubes and hence the existence of at least $k$ distinct $n$-approximate fixed points.
\end{theorem}
\begin{corollary} \label{cor_cupbd}
	Let $\bar L(f)$ be the cubical Lefschetz number of a continuous self-map $f$ on a digital image $X.$ Then, $|\bar L(f)| \leq \# X.$
\end{corollary}
\begin{figure}
	\[
	\begin{tikzpicture}[scale=.35]
	
	\foreach \y in {1,2} {
		\foreach \x in {3,4,7,8} {
			\node[vertex] at (\x,\y) {};
		}
	}
	\foreach \y in {3,4} {
		\foreach \x in {1,4,7,10} {
			\node[vertex] at (\x,\y) {};
		}
	}
	\foreach \y in {5,7} {
		\foreach \x in {1, 3,4,...,8, 10} {
			\node[vertex] at (\x,\y) {};
		}
	}
	\foreach \x in {1,2,3,8,9,10} {
		\node[vertex] at (\x,6) {};
	}
	\foreach \x in {1,10} {
		\node[vertex] at (\x,8) {};
	}
	\draw(1,3) -- (1,8);
	\draw(10,3) -- (10,8);
	\draw(3,5) rectangle (8,7);
	\draw(1,6) -- (3,6);
	\draw(8,6) -- (10,6);
	\draw(4,1) -- (4,5);
	\draw(7,1) -- (7,5);
	\draw(3,1) rectangle (4,2);
	\draw(7,1) rectangle (8,2);
	\end{tikzpicture}
	\qquad
	\begin{tikzpicture}[scale=.35]
	\node at (1,1) {};
	\foreach \y in {1,2} {
		\foreach \x in {3,4,7,8} {
			\node[vertex] at (\x,\y) {};
		}
	}
	\foreach \y in {3,4} {
		\foreach \x in {4,7} {
			\node[vertex] at (\x,\y) {};
		}
	}
	\foreach \y in {5,7} {
		\foreach \x in {3,4,...,8} {
			\node[vertex] at (\x,\y) {};
		}
	}
	\foreach \x in {3,8} {
		\node[vertex] at (\x,6) {};
	}
	\draw(3,5) rectangle (8,7);
	\draw(4,1) -- (4,5);
	\draw(7,1) -- (7,5);
	\draw(3,1) rectangle (4,2);
	\draw(7,1) rectangle (8,2);
	\end{tikzpicture}
	\qquad
	\begin{tikzpicture}[scale=.35]
	\node at (1,1) {};
	\foreach \y in {5,7} {
		\foreach \x in {3,4,...,8} {
			\node[vertex] at (\x,\y) {};
		}
	}
	\foreach \x in {3,8} {
		\node[vertex] at (\x,6) {};
	}
	
	\draw(3,5) rectangle (8,7);
	\end{tikzpicture}
	\]
	\caption{A digital image $X$ (left), together with its minimal reduction (middle) with respect to strong homotopy equivalence, and its minimal reduction (right) with respect to homotopy equivalence.}
	\label{robot}
\end{figure}
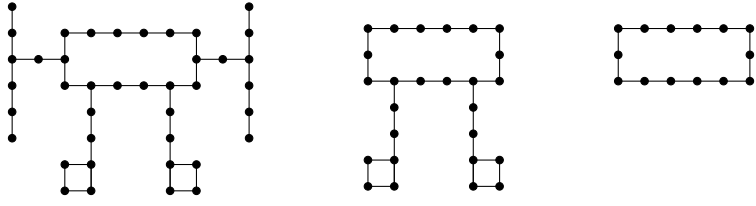

\subsection{Commutativity Property of the Lefschetz Numbers}
This part comprises the study of commutative property of both the simplicial Lefschetz number and the cubical Lefschetz number, and the study of reduction of an image $X$ with respect to (strong) homotopy equivalences. We also briefly discuss some possible applications of these properties to image thinning operations.
\begin{theorem}\label{commutative}
	Let $X$ and $Y$ be digital images, and $f : X \longrightarrow Y$ and $g : Y \longrightarrow X$ be continuous. Then $L(f \circ g) = L(g \circ f)$. If $X$ and $Y$ are subsets of $\Z^n$ with $c_1$ adjacency, then $\bar L(f\circ g) = \bar L(g\circ f)$.
\end{theorem}
\begin{proof}
	Since $f_{*,q}$ and $g_{*,q}$ are square matrices, and by the definition of $L(f)$ and the fact that the trace of a composition of square matrices is commutative (that is, $\tr(A \circ B) = \tr(B \circ A)$), we have
	\begin{align*}
	L(f \circ g) &= \sum_{q=0}^n (-1)^q \tr((f \circ g)_{*,q})\\
	&= \sum_{q=0}^n (-1)^q \tr(f_{*,q} \circ g_{*,q})\\
	&= \sum_{q=0}^n (-1)^q \tr(g_{*,q} \circ f_{*,q})\\
	&= \sum_{q=0}^n (-1)^q \tr((g \circ f)_{*,q}) = L(g \circ f),
	\end{align*}
	where $f_{*,q}, g_{*,q} : H_q(X) \longrightarrow H_q(X).$ Now, following the same argument shows that $\bar L(f\circ g) = \bar L(g\circ f)$.
\end{proof}
\begin{theorem}\label{homoequiv}
	Let $X$ and $Y$ be digital images both with $c_1$-adjacency. Let $f : X \longrightarrow X$ be a mapping and $g : X \longrightarrow Y$ be a digital homotopy equivalence with homotopy inverse $h : Y \longrightarrow X$. Then $\bar L(f) = \bar L(g \circ f \circ h).$
\end{theorem}
\begin{proof}
	Since $g : X \longrightarrow Y$ is digitally homotopy equivalence with homotopy inverse $h : Y \longrightarrow X$. By Theorem \ref{equitype}, we have $h \circ g \simeq id_X$ and $g \circ h \simeq id_Y.$ Now, by Theorem \ref{commutative}, we obtain
	$h \circ g \simeq id_X$ and $h \circ g \circ f \simeq id_X \circ f \simeq f.$ Hence, we obtain $\bar L(h \circ g \circ f) = \bar L(g \circ f \circ h) = \bar L(f).$
\end{proof}
The proof of the following result is similar to that of the preceding theorem, hence we omitted it.
\begin{theorem}\label{homostrngequiv}
	Let $X$ and $Y$ be digital images. Let $f : X \longrightarrow X$ be a mapping and $g : X \longrightarrow Y$ be a digitaly strong homotopy equivalence with strong homotopy inverse $h : Y \longrightarrow X$. Then $L(f) = L(g \circ f \circ h).$
\end{theorem}
The following is similar to the idea of ``image thinning" in digital topology. This has a nice interpretation in terms of digital ``thinning" operations.
\begin{definition}
	A digital image $Y$ is said to be a ``(strong) reduction" of a digital image $X$ if the images $X$ and $Y$ are (strong) homotopy equivalent and $\# Y < \# X.$ If there is no (strong) reduction of cardinality less than $Y$, we say $Y$ is a ``minimal (strong) reduction.''
\end{definition}
%
Note that, any digital image $X$ has a minimal reduction $X'$ and a minimal strong reduction $X^*$, and we easily have the following relation:
$$\#X' \leq \#X^* \leq \#X.$$
Proposition \ref{homostrngequiv} implies that computing $L(f)$ on $X$ is equivalent to computing $L(f)$ on $X^*$. Similarly, Proposition \ref{homoequiv} implies that computing $\bar L(f)$ on $X$ is equivalent to computing $\bar L(f)$ on $X'.$ This is interesting, as it is easier to compute $L(f)$ and $\bar L(f)$ on smaller spaces (i.e $X^*$ and $X'$ respectively) than on the bigger space $X$.
\begin{example}\label{ex5}
	Let $X$ be the image in Figure $\ref{robot}$ and $f$ be a self-map on $X$. Then 
	$$\bar L(f) \in \{0,1,2\}.$$ 
	This is true by Proposition \ref{prop} because the minimal reduction of $X$ (i.e. $X'$) is a digital cycle. Whereas, the minimal strong reduction of $X$ (i.e. $X^*$) is not a digital cycle; since it has $3$ loops. Therefore, it is possible that there is some map $f: X \longrightarrow X$ with $L(f) \not\in \{0,1,2\}$. In fact, taking the self map $f$ to be the identity map, we get $L(id) = \chi(X) = -2.$ 
\end{example}

\subsection{Spectrum of the Lefschetz Numbers}
Recall that, the concept of fixed point spectrum of an image $X$ that gives the set of all numbers that can appear as the number of fixed points for some continuous self-map was introduced recently as follows:
\begin{definition}\cite{BoxerSta19fixed}
	Let $X$ be a digital image. Then the fixed point spectrum of $X$ is defined as:
	\[F(X) = \{\# \textup{Fix}(f) \, | \, f : X \longrightarrow X \mbox{ is continuous}\}.\]
\end{definition}
But it turns out that the computation of the fixed point spectrum for an image $X$ is difficult. This is because not much is known about how the set $F(X)$ changes when we change the image $X.$ One of the only few things we know is that, whenever $A$ is a retract of an image $X$ we then have $F(A) \subseteq F(X)$.

Now, we will present the definitions of both the spectrum of simplicial Lefschetz number which we denote as $\mathcal L(X)$ and the spectrum of cubical Lefschetz number denoted by $\mathcal{\bar L}(X)$ as in the following:
\begin{definition}
	Let $X$ be a digital image. We define the spectrum of simplicial Lefschetz number and the spectrum of cubical Lefschetz number as:
	\[\mathcal L(X) = \{L(f) : f \mbox{ is a continuous self-map on } X\}\] and \[\mathcal{\bar L}(X) = \{\bar L(f) : f \mbox{ is a continuous self-map on } X\}\]
	respectively.
\end{definition}
\begin{example}
	Let $X$ be the image in Figure \ref{robot} and $f$ be a self-map on $X$. Then $$\mathcal{\bar L}(X) = \{0,1,2\}.$$ This is true from Proposition \ref{prop} and Example \ref{ex5}.
\end{example}
By the commutativity property, we have
\begin{proposition}
	The spectrum of the cubical Lefschetz number is a homotopy-type invariant.
\end{proposition}
\begin{proof}
	By Theorem \ref{bhithm}, \ref{homoequiv} and Definition \ref{equitype}, we have $\bar L(f) = \bar L(g)$ whenever $f$ and $g$ are homotopy-type equivalences. Hence the result follows as $f$ and $g$ are arbitrary. 
\end{proof}
\begin{proposition}
	The spectrum of the simplicial Lefschetz number is a strong homotopy-type invariant.
\end{proposition}
Unlike the fixed point spectrum which is not a homotopy-type invariant. We obtain that the spectrum of cubical Lefschetz number is homotopy-type invariant while that of simplicial Lefschetz number is strong homotopy-type invariant. That is, if $X$ and $Y$ are homotopy equivalent, then $\bar{\mathcal L}(X) = \bar{\mathcal L}(Y)$ and if $X$ and $Y$ are strong homotopy equivalent, then $\mathcal L(X) = \mathcal L(Y).$ respectively. This is particularly more interesting fact, because it means in general we can reduce an image by homotopy equivalence or strong homotopy equivalence to obtain a much simpler image without changing the Lefschetz theory.
\begin{remark}
	The cardinality of an image $X$ is an upper bound for both simplicial and cubical Lefschetz spectrum. (i.e. $m \leq \#X$ for all $m \in \mathcal{L}(X)$ or $m \in \mathcal{\bar L}(X)$). Moreover, $- \#X \leq m$ for all $m \in \mathcal{L}(X)$ or $m \in \mathcal{\bar L}(X)$, i.e. The negative cardinality of an image $X$ is a lower bound for both simplicial and cubical Lefschetz spectrum. Hence, $\mathcal L(f), \bar{\mathcal L}(X) \subseteq \{-\#X, \ldots, \#X\}.$ This follows from Corollary \ref{cor_supbd} and \ref{cor_cupbd} respectively.
\end{remark}

\subsection{Relation between the Lefschetz Numbers, $S_{a^n}(f)$ and $S_a^*(f)$.}
In this part, we will introduce briefly a relationship between the defined Lefschetz numbers with the following concepts. This was motivated by the notion of homotopy fixed point spectrum in \cite{BoxerSta19fixed}.
\begin{definition}
	Let $f,g : X \longrightarrow X$ be any continuous maps on $X$. Then 
	\begin{itemize}
		\item The homotopy $n$-approximate fixed point spectrum of $f$ is denoted and defined as:
		\[S_{a^n}(f) = \{\nAFix{n}(g) \mid g \simeq f \} \subseteq \{0, 1,  \ldots ,\#X\},\]
		\item The strong homotopy approximate fixed point spectrum of $f$ is denoted and defined as:
		\[S_a^*(f) = \{\AFix(g) \mid g \simeq^* f \} \subseteq \{0, 1,  \ldots ,\#X\}.\]
	\end{itemize}
	where $\AFix(g)$ and $\nAFix{n}(g)$ denote the sets of approximate fixed points, and $n$-approximate fixed points, respectively.
\end{definition}
\begin{remark}\hfill
	\begin{itemize}
		\item $S_{a^n}(f)$ is a homotopy invariant for any continuous map $f$ on $X$.
		\item $S_a^*(f)$ is a strong homotopy invariant for any continuous map $f$ on $X$.
	\end{itemize}
\end{remark}
\begin{proposition}
	Let $X \subset \mathbb{Z}^n$ be a digital image with $c_1$-adjacency and $n \leq 3$. If $S_{a^n}(f)$ is the homotopy $n$-approximate fixed point spectrum of $X$, then $|\bar L(f)|$ is a lower bound for $S_{a^n}(f).$
\end{proposition}
\begin{proof}
	This follows from Theorem \ref{no_nafpt}, since $|\bar L(f)|$ is the least number of $n$-approximate fixed point that any continuous map $f$ on $X$ can attain.
\end{proof}
Similar statements also hold concerning $L(f)$ and the strong homotopy spectrum of approximate fixed points $S_a^*(f)$.

\section{Conclusion}
In this paper, we utilized the classical simplicial and cubical homology theories and defined (in the usual way) the simplicial Lefschetz number and the cubical Lefschetz number. We showed that the simplicial Lefschetz number $L(f)$ is strong homotopy invariant and has a $1$-approximate fixed point theorem. While the cubical Lefschetz number $\bar L(f)$ is homotopy invariant and has an  $n$-approximate fixed point theorem.

Consequently, we can conclude that the fixed point theorem for $L(f)$ seems better, but has a worse homotopy invariance as it requires a stronger homotopy. On the other hand, the fixed point theorem for the cubical Lefschetz number $\bar L(f)$ is worse because it may fail to provide us with a fixed or approximate fixed point, but we found out that it has a better homotopy invariance property.

We have also demonstrated with some illustrative examples that the two Lefschetz numbers are essentially different. In particular, we have shown an example where $L(f)$ is zero while $\bar L(f)$ is nonzero. We believed that, these differences, advantages and disadvantages would demonstrate the usefulness of both Lefschetz numbers in homology of digital images, for instance in detecting fixed or approximate fixed points.

Moreover, we have discussed the commutativity property of the Lefschetz numbers, which implies that the Lefschetz numbers are invariant under (strong) homotopy equivalence. Also, we have highlighted its possible applications to image thinning in digital topology. We also introduced an invariant; namely "the spectrum of Lefschetz number" to help us find all the possible values of the Lefschetz numbers for all self-maps on a given digital image.

Finally, we presented a nice relation between the digital Lefschetz numbers with the number of approximate fixed points. We showed that $|\bar L(f)|$ is a lower bound to the number of $n$-approximate fixed points, $|L(f)|$ is a lower bound to the number of approximate fixed points and $\#X$ is an upper bound to both Lefschetz spectrum.


%

\subsection*{Acknowledgments}
The authors acknowledge financial support provided by the Center of Excellence in Theoretical and Computational Science (TaCS-CoE), KMUTT. The first author was supported by the ``Petchra Pra Jom Klao Ph.D. Research Scholarship from King Mongkut's University of Technology Thonburi" (Grant No.: 35/2017).

\bibliographystyle{plain} 
\bibliography{reference}         

\begin{thebibliography}{10}

\bibitem{AbdullahiAzam17A}
Muhammad~Sirajo Abdullahi and Akbar Azam.
\newblock ${L}$-fuzzy fixed point theorems for ${L}$-fuzzy mappings via
  $\beta_{F_L}$-admissible with applications.
\newblock {\em Journal of Uncertainty Analysis and Applications}, 5(2):1--13,
  2017.

\bibitem{AbdullahiPoom18}
Muhammad~Sirajo Abdullahi and Poom Kumam.
\newblock Partial $b_v(s)$-metric spaces and fixed point theorems.
\newblock {\em Journal of Fixed Point Theory and Applications}, 20(3):113,
  2018.

\bibitem{Arslanetal08homology}
H~Arslan, I~Karaca, and A~Oztel.
\newblock Homology groups of $n$-dimensional digital images, xxi.
\newblock In {\em Turkish National Mathematics Symposium, B}, volume~1,
  page~13, 2008.

\bibitem{Azametal09}
Akbar Azam, Muhammad Arshad, and Ismat Beg.
\newblock Fixed points of fuzzy contractive and fuzzy locally contractive maps.
\newblock {\em Chaos, Solitons and Fractals}, 42(5):2836--2841, 2009.

\bibitem{Banach}
Stefan Banach.
\newblock Sur les operations dans les ensembles abstraits et leur application
  aux equations integrales.
\newblock {\em Fund. Math}, 3(1):133--181, 1922.

\bibitem{Bertrand94simple}
Giles Bertrand.
\newblock Simple points, topological numbers and geodesic neighborhoods in
  cubic grids.
\newblock {\em Pattern recognition letters}, 15(10):1003--1011, 1994.

\bibitem{Boxer94digitally}
Laurence Boxer.
\newblock Digitally continuous functions.
\newblock {\em Pattern Recognition Letters}, 15(8):833--839, 1994.

\bibitem{Boxer99classical}
Laurence Boxer.
\newblock A classical construction for the digital fundamental group.
\newblock {\em Journal of Mathematical Imaging and Vision}, 10(1):51--62, 1999.

\bibitem{Boxer16generalized}
Laurence Boxer.
\newblock Generalized normal product adjacency in digital topology.
\newblock {\em arXiv preprint arXiv:1608.03204}, 2016.

\bibitem{Boxeretal16digital}
Laurence Boxer, Ozgur Ege, Ismet Karaca, Jonathan Lopez, and Joel Louwsma.
\newblock Digital fixed points, approximate fixed points, and universal
  functions.
\newblock {\em Applied General Topology}, 17(2):159--172, 2016.

\bibitem{Boxeretal11topological}
Laurence Boxer, Ismet Karaca, and A~Oztel.
\newblock Topological invariants in digital images.
\newblock {\em J. Math. Sci. Adv. Appl.}, 11(2):109--140, 2011.

\bibitem{BoxerSta19fixed}
Laurence Boxer and P~Christopher Staecker.
\newblock Fixed poin sets in digital topology, 1.
\newblock {\em Applied General Topology}, 21(1):87--110, 2020.

\bibitem{Brouwer11abbildung}
Luitzen Egbertus~Jan Brouwer.
\newblock {\"U}ber abbildung von mannigfaltigkeiten.
\newblock {\em Mathematische Annalen}, 71(1):97--115, 1911.

\bibitem{Brown71lefschetz}
Robert~F Brown.
\newblock {\em The {L}efschetz fixed point theorem}.
\newblock Scott, Foresman, Glenview, Ill., 1971.

\bibitem{EgeKaraca13lefschetz}
Ozgur Ege and Ismet Karaca.
\newblock Lefschetz fixed point theorem for digital images.
\newblock {\em Fixed Point Theory and Applications}, 2013(1):253, 2013.

\bibitem{Han07digital}
Sang-Eon Han.
\newblock Digital fundamental group and euler characteristic of a connected sum
  of digital closed surfaces.
\newblock {\em Information Sciences}, 177(16):3314--3326, 2007.

\bibitem{JamilAli19digital}
Samira~Sahar Jamil and Danish Ali.
\newblock Digital {H}urewicz theorem and digital homology theory.
\newblock {\em Turkish Journal of Mathematics (Accepted), arXiv:1902.02274},
  2020.

\bibitem{KaracaEge12cubical}
I~Karaca and O~Ege.
\newblock Cubical homology in digital image, int.
\newblock {\em Journal of Information and Computer Science}, 1(7):178--187,
  2012.

\bibitem{Khalimsky87motion}
Efim Khalimsky.
\newblock Motion, deformation, and homotopy in finite spaces.
\newblock In {\em Proceedings IEEE International Conferences on Systems, Man,
  and Cybernetics}, pages 227--234, 1987.

\bibitem{KongRosen96topological}
T~Yung Kong and Azriel Rosenfeld.
\newblock {\em Topological algorithms for digital image processing}, volume~19.
\newblock Elsevier, New York, 1996.

\bibitem{Lefschetz26intersections}
Solomon Lefschetz.
\newblock Intersections and transformations of complexes and manifolds.
\newblock {\em Transactions of the American Mathematical Society}, 28(1):1--49,
  1926.

\bibitem{Luptonetal19fundamental}
Gregory Lupton, John Oprea, and Nicholas Scoville.
\newblock A fundamental group for digital images.
\newblock {\em arXiv preprint arXiv:1906.05976}, 2019.

\bibitem{Naber80topological}
Gregory~L Naber.
\newblock {\em Topological methods in Euclidean spaces}.
\newblock CUP Archive, New York, 1980.

\bibitem{Nadler}
Sam~B Nadler~Jr.
\newblock Multi-valued contraction mappings.
\newblock {\em Pacific J. Math}, 30(2):475--488, 1969.

\bibitem{Nielsen27untersuchungen}
Jakob Nielsen.
\newblock Untersuchungen zur {T}opologie der geschlossenen zweiseitigen
  fl{\"a}chen.
\newblock {\em Acta Mathematica}, 50(1):189--358, 1927.

\bibitem{Rosen79digital}
Azriel Rosenfeld.
\newblock Digital topology.
\newblock {\em The American Mathematical Monthly}, 86(8):621--630, 1979.

\bibitem{Rosen86continuous}
Azriel Rosenfeld.
\newblock ‘continuous’ functions on digital pictures.
\newblock {\em Pattern Recognition Letters}, 4(3):177--184, 1986.

\bibitem{SintuKumam11}
Wutiphol Sintunavarat and Poom Kumam.
\newblock Common fixed point theorems for a pair of weakly compatible mappings
  in fuzzy metric spaces.
\newblock {\em Journal of Applied Mathematics}, 2011:1--14, 2011.

\bibitem{Staecker20digital}
P~Christopher Staecker.
\newblock Digital homotopy relations and digital homology theories.
\newblock {\em (Submitted)}, pages 1--28, 2020.

\end{thebibliography}

\vspace{0.5cm}

\end{document}